\documentclass[12pt, reqno]{amsart}

\usepackage{amsfonts, amsmath, amssymb, amsthm}
\usepackage{geometry}
\usepackage{enumitem}
\usepackage{hyperref}
\usepackage{graphicx}
\graphicspath{{figures/}}
\usepackage{pgfplots}
\pgfplotsset{compat=1.18}
\usepackage{tikz}
\usetikzlibrary{calc,hobby,patterns,datavisualization}
\usepackage{pdfpages}

\newtheorem{theorem}{Theorem}[section]
\newtheorem{lemma}[theorem]{Lemma}
\newtheorem{proposition}[theorem]{Proposition}
\newtheorem{corollary}[theorem]{Corollary}
\newtheorem{definition}[theorem]{Definition}

\newtheorem{remark}[theorem]{Remark}
\newtheorem{example}[theorem]{Example}

\title[CROSS WORK ON A SYMMETRIC GROUP]{CROSS WORK INDUCED BY TRANSFORMATIONS ON THE SYMMETRIC GROUP AND RELATED TRANSFORMATION SEMIGROUPS}
\date{November 2025}

\author[N.~E.~Udo, V.~A.~Owai and M~O.~Francis,]{Nsibiet~E.~Udo$^\ast$, Victor~A.~Owai, and Moses~O.~Francis}
\thanks{$^\ast$ Corresponding author}

\begin{document}
\begin{abstract}
	We introduce a new permutation statistic, called the \emph{cross work}, on four important subsemigroups of the partial transformation semigroup $\mathcal{PT}_n$: the symmetric group $\mathcal{S}_n$, the semigroup of singular transformations $\mathrm{Sing}_n$, the symmetric inverse semigroup $\mathcal{I}_n$, and the semigroup $\mathcal{POI}_n$ of partial order-preserving injective transformations. The cross work is defined via local distance interactions between adjacent positions and their images. Fundamental invariance properties are established, including inversion invariance under a natural symmetry condition and a characterization of conjugations preserving the statistic. Using elementary combinatorial methods, we derive closed formulas for the total and average cross work on $\mathcal{S}_n$. We also determine the exact minimum and maximum values of the cross work, showing that the maximal value grows quadratically with $n$.
\end{abstract}

\maketitle

\section{Introduction}
Permutation statistics play a central role in algebraic combinatorics, representation theory, and the theory of transformation semigroups. Classical statistics such as inversions, descents, the major index, and displacement encode diverse structural properties of permutations and often admit rich algebraic and enumerative interpretations.

In the context of transformation semigroups, quantitative notions of work were systematically introduced by East and McNamara~\cite{EastMcNamara2011}, motivated by earlier ideas attributed to Lavers~\cite{Lavers2004}. Their work provided a rigorous framework for numerical invariants that measure the effort required to move points under transformations in full and partial transformation semigroups, thereby opening a new quantitative direction in semigroup theory. Closely related displacement-based measures for permutations and transformations have also been investigated in the combinatorial literature, notably by Daly and Vojtěchovský~\cite{DalyVojtechovsky2009}, as well as Laradji and Umar \cite{LaradjiUmar2004}, who studied order-preserving partial transformations from a combinatorial viewpoint.

Following this foundational work, several authors have extended the notion of work and related invariants to more structured algebraic settings. In particular, Kehinde et al.~\cite{Kehinde.et.al.2020} examined numerical work and power on the symmetric group, providing explicit computations and graphical interpretations. Further generalizations were developed for $m$-topological transformation semigroup spaces and their subclasses, including regular, closed, and clopen spaces \cite{FrancisAdenijiMogbonju2023,FrancisAdeniji2025}. These studies demonstrate that work-type invariants interact naturally with additional algebraic and topological structure, leading to rich combinatorial behavior.

The study of permutation statistics more broadly has a long and influential history in algebraic combinatorics. Classical statistics such as inversions, descents, major index, and displacement have been extensively studied; see, for example, the works of MacMahon \cite{MacMahon1915}, Stanley \cite{Stanley2012}, and Bóna \cite{Bona2012}. Distance-based statistics, including total displacement and Spearman-type measures, have proven particularly useful in connecting permutation theory with geometry and probability \cite{DiaconisGraham1977,DalyVojtechovsky2009}. Comprehensive background on transformation semigroups and their algebraic properties may be found in the works of Howie \cite{Howie1995} and Umar \cite{Umar2016}.

The symmetric group $\mathcal{S}_n$, as the group of units of the full transformation semigroup $\mathcal{T}_n$, and $\mathrm{Sing}_n = \mathcal{T}_n \setminus \mathcal{S}_n$, its unique maximal ideal (commonly known as the semigroup of singular transformations), provide a natural and mathematically rich setting in which to investigate such numerical invariants.

In this paper, we introduce a new permutation statistic on $\mathcal{S}_n$, called the \emph{cross work}. Roughly speaking, the cross work of a permutation measures how far the images of adjacent positions deviate from their original locations, taking into account both left and right interaction effects. Unlike classical displacement-based measures studied in \cite{EastMcNamara2011,DalyVojtechovsky2009,FrancisAdeniji2025}, the cross work captures local adjacency interactions while remaining globally computable and amenable to closed-form analysis. Furthermore, in contrast to global displacement measures such as those in \cite{DalyVojtechovsky2009}, the cross work is local in nature: it aggregates the effort contributed by each adjacent pair separately, so that its extremal permutations admit an explicit block description (Theorem~\ref{thm:max-count}) and its minimisers are counted by a Fibonacci recurrence (Proposition~\ref{prop:mincount}).

Section~\ref{sec:prelim-defn} introduces the cross work statistic and presents basic examples, together with the notions of total and average cross work. Section~\ref{sec:main_results} is devoted to invariance properties, including inversion invariance (Theorem~\ref{thm:inversion-invariance}) and the characterisation of conjugations preserving the statistic Theorem~\ref{thm:invariance_condition}). It also contains the main results on the total and average cross work on $\mathcal{S}_n$ (Theorem~\ref{thm:Sn} and Corollary~\ref{cor:avg-crosswork}, respectively), and on the semigroup of singular transformations $\mathrm{Sing}_n$ (Theorem~\ref{thm:Sing}), the symmetric inverse semigroup of injective partial transformations $\mathcal{I}_n$ (Theorem~\ref{thm:I}), and the monoid of partial order-preserving injections $\mathcal{POI}_n$ (Theorem~\ref{thm:POI}). We conclude with remarks highlighting connections to classical permutation statistics and outlining directions for further study.

\section{Preliminaries and Definitions}\label{sec:prelim-defn}
Throughout, let \([n] = \{1, 2, \dots, n\}\) be a finite set with \(n\) elements. A \textit{full transformation} is a map \(\alpha\colon [n]\to[n]\); the set \(\mathcal T_n\) of all such maps is the full transformation semigroup under composition, with identity the identity map. A \textit{partial transformation} on \([n]\) is a map \(\alpha\colon \operatorname{dom}(\alpha)\subseteq[n]\to[n]\); the set \(\mathcal{PT}_n\) of all partial transformations forms the partial transformation semigroup under composition. A \textit{finite transformation semigroup on \([n]\)} is any subsemigroup \(\mathcal G\subseteq\mathcal{PT}_n\); thus every \(\alpha\in\mathcal G\) is a partial transformation on \([n]\), and \(\operatorname{dom}(\alpha)\) is defined as below. We denote by $\mathcal{S}_n$ the symmetric group on $[n]$, i.e. the subgroup of \(\mathcal{PT}_n\) of all permutations of \([n]\). Permutations are written in one-line notation.

\begin{definition}\label{def:local_crosswork}
Let $n \ge 2$ be an integer and let $\mathcal G\subseteq \mathcal{PT}_n$ be a finite transformation semigroup on $[n]$. For a partial transformation $\alpha \in\mathcal G$, denote its domain by $\operatorname{dom}(\alpha)$. For each $1 \le i \le n-1$, define the \emph{left distance} and \emph{right distance} of $\alpha\in\mathcal G$ over adjacent positions $(i, i+1)$ by
\[
\delta_{i, \alpha}^{L} = |i - \alpha(i+1)|, \qquad
\delta_{i, \alpha}^{R} = |(i+1) - \alpha(i)|
\]
when $i, i+1 \in \operatorname{dom}(\alpha)$, and zero otherwise; with this convention the sum below may be taken over all \(1\le i\le n-1\). The (local) \emph{cross work} done by $\alpha\in\mathcal G$ over the adjacent pair $(i, i+1)$ is defined to be
\[
w_i^\times = \delta_{i, \alpha}^L + \delta_{i, \alpha}^R,
\]
and the \emph{cross work} of $\alpha$ is
\[
w^\times(\alpha) = \sum_{\substack{1 \le i \le n-1 \\ i,i+1 \in \operatorname{dom}(\alpha)}} w_i^\times = \sum_{i=1}^{n-1} \left( \delta_{i, \alpha}^{L} + \delta_{i, \alpha}^{R} \right).
\]
\end{definition}

\begin{example}
Let $\alpha = 231\in\mathcal{S}_3$, so that \(\alpha(1)=2\), \(\alpha(2)=3\), \(\alpha(3)=1\). A direct computation gives
\[
\delta_{1, \alpha}^{L} = |1-\alpha(2)| = |1-3| = 2, \quad
\delta_{1, \alpha}^{R} = |2-\alpha(1)| = |2-2| = 0,
\]
\[
\delta_{2, \alpha}^{L} = |2-\alpha(3)| = |2-1| = 1, \quad
\delta_{2, \alpha}^{R} = |3-\alpha(2)| = |3-3| = 0.
\]
Hence,
\[
w^\times(\alpha) = \sum_{i=1}^{2} \left(\delta_{i, \alpha}^{L} + \delta_{i, \alpha}^{R}\right) = (2+0) + (1+0) = 3.
\]
Table~\ref{tab:crosswork} lists the cross work values for all permutations in $\mathcal{S}_3$.

\begin{table}[h]
\centering
\renewcommand{\arraystretch}{1.2}
\setlength{\tabcolsep}{12pt}
\caption{Values of the cross work $w^\times(\alpha)$ for all permutations $\alpha \in \mathcal{S}_3$}
\label{tab:crosswork}
\begin{tabular}{c|c}
\hline
$\alpha \in \mathcal{S}_3$ & $w^\times(\alpha)$ \\
\hline
123 & 4 \\
132 & 3 \\
213 & 3 \\
231 & 3 \\
312 & 3 \\
321 & 4\\
\hline
\end{tabular}
\end{table}
\end{example}

\begin{definition}\label{def:total_crosswork}
For a fixed integer $n \ge 2$, let $\mathcal G$ be a finite transformation semigroup on $[n]$. The \emph{total cross work} on $\mathcal G$ is defined by
\[
w^\times(\mathcal G) = \sum_{\alpha \in\mathcal G} w^\times(\alpha).
\]
The \emph{average cross work} on $\mathcal G$ is
\[
\overline{w}^\times(\mathcal G) = \frac{1}{|\mathcal G|} \sum_{\alpha\in\mathcal G} w^\times(\alpha) = \frac{w^\times(\mathcal G)}{|\mathcal G|}.
\]
\end{definition}

\begin{example}
For $n = 3$, we see from Table~\ref{tab:crosswork} that
\[
w^\times(\mathcal{S}_3) = \sum_{\alpha \in \mathcal{S}_3} w^\times(\alpha) = 4 + 3 + 3 + 3 + 3 + 4 = 20.
\]
\end{example}

\section{Main Results}\label{sec:main_results}
In this section, we present the main results concerning cross work on $\mathcal{S}_n$, $\mathcal{T}_n$, $\mathrm{Sing}_n$, $\mathcal{I}_n$, and $\mathcal{POI}_n$.

\begin{definition}
Let $\alpha, \beta \in \mathcal{S}_n$. We say that $\alpha$ and $\beta$ are \emph{cross-work invariant} if $w^\times(\alpha) = w^\times(\beta)$.
\end{definition}

We begin with a sufficient condition under which cross work is preserved by inversion. 

\begin{lemma}\label{lem:equivalence_of_distances}
Let $n \ge 2$ and let $\alpha \in \mathcal{S}_n$. If $\alpha$ transposes $1$ and $n$ (i.e., $\alpha(1)=n$ and $\alpha(n)=1$), then the right distances of $\alpha^{-1}$ are exactly the left distances of $\alpha$, and vice versa. More precisely,
\[
\{\delta_{i, \alpha}^L \mid 1 \le i \le n-1\} = \{\delta_{j, \alpha^{-1}}^R \mid 1 \le j \le n-1\}
\]
and
\[
\{\delta_{i, \alpha}^R \mid 1 \le i \le n-1\} = \{\delta_{j, \alpha^{-1}}^L \mid 1 \le j \le n-1\}.
\]
\end{lemma}

\begin{proof}
Let $1 \le i \le n-1$. Since $\alpha(1)=n$ and $\alpha$ is a bijection, we have $1 \le \alpha(i+1) \le n-1$. Furthermore, $\delta_{i, \alpha}^L = \delta_{\alpha(i+1), \alpha^{-1}}^R$. Setting $j = \alpha(i+1)$, we see that $\delta_{i, \alpha}^L \in \{\delta_{j, \alpha^{-1}}^R \mid 1 \le j \le n-1\}$. This proves the first inclusion.

Now let $1 \le j \le n-1$. Notice that $\alpha^{-1}$ also transposes $1$ and $n$, so $2 \le \alpha^{-1}(j) \le n$. Subtracting $1$, we get $1 \le \alpha^{-1}(j)-1 \le n-1$. Moreover, $\delta_{j, \alpha^{-1}}^R = \delta_{\alpha^{-1}(j)-1, \alpha}^L$. Setting $i = \alpha^{-1}(j)-1$, we see that $\delta_{j, \alpha^{-1}}^R \in \{\delta_{i, \alpha}^L \mid 1 \le i \le n-1\}$. This proves the reverse inclusion, hence the first equality.

Since $(\alpha^{-1})^{-1} = \alpha$, the same argument applied to $\alpha^{-1}$ yields the second equality.
\end{proof}

\begin{theorem}\label{thm:inversion-invariance}
Let $\alpha \in \mathcal{S}_n$ be such that $\alpha(1)=n$ and $\alpha(n)=1$. Then $\alpha$ and $\alpha^{-1}$ are cross-work invariant.
\end{theorem}

\begin{proof}
By Lemma~\ref{lem:equivalence_of_distances},
\[
\sum_{i=1}^{n-1} \delta_{i, \alpha}^L = \sum_{i=1}^{n-1} \delta_{i, \alpha^{-1}}^R, \quad
\sum_{i=1}^{n-1} \delta_{i, \alpha}^R = \sum_{i=1}^{n-1} \delta_{i, \alpha^{-1}}^L.
\]
Therefore,
\[
w^\times(\alpha) = \sum_{i=1}^{n-1} (\delta_{i, \alpha}^L + \delta_{i, \alpha}^R)
= \sum_{i=1}^{n-1} (\delta_{i, \alpha^{-1}}^R + \delta_{i, \alpha^{-1}}^L)
= w^\times(\alpha^{-1}).
\]
\end{proof}

\begin{example}
Let $\alpha = 53421 \in \mathcal{S}_5$. Then $\alpha^{-1} = 54231$, and a direct computation gives $w^\times(\alpha) = w^\times(\alpha^{-1}) = 14$.
\end{example}

Unlike classical displacement statistics studied in \cite{EastMcNamara2011,FrancisAdenijiMogbonju2023}, cross work is not invariant under arbitrary conjugation in general. The following result characterizes when conjugation preserves cross work.

\begin{proposition}\label{prop:reversal-invariance}
Let $\tau \in \mathcal{S}_n$ be the reversal permutation defined by \(\tau(i) = n+1-i\) for \(1 \le i \le n\). Then for any $\pi \in \mathcal{S}_n$,
\[
w^\times(\tau \pi \tau^{-1}) = w^\times(\pi).
\]
\end{proposition}

\begin{proof}
Since $\tau^{-1} = \tau$, write $\sigma = \tau \pi \tau$. For $1 \le i \le n-1$,
\[
\sigma(i+1) = \tau(\pi(\tau(i+1))) = \tau(\pi(n-i)) = n+1 - \pi(n-i),
\]
\[
\sigma(i) = \tau(\pi(\tau(i))) = \tau(\pi(n+1-i)) = n+1 - \pi(n+1-i).
\]
Hence,
\[
|i - \sigma(i+1)| = |(n+1-i) - \pi(n-i)|, \quad
|(i+1) - \sigma(i)| = |(n-i) - \pi(n+1-i)|.
\]
Thus,
\[
\delta_{i,\sigma}^L + \delta_{i,\sigma}^R = |(n+1-i) - \pi(n-i)| + |(n-i) - \pi(n+1-i)|.
\]
Letting $j = n-i$, as $i$ runs from $1$ to $n-1$, so does $j$. Therefore,
\[
\delta_{i,\sigma}^L + \delta_{i,\sigma}^R = |j+1 - \pi(j)| + |j - \pi(j+1)| = \delta_{j,\pi}^L + \delta_{j,\pi}^R.
\]
Summing over all indices gives $w^\times(\sigma) = w^\times(\pi)$.
\end{proof}

The preceding proposition shows that cross work admits a nontrivial global symmetry. We now show that this symmetry is essentially unique.

\begin{theorem}\label{thm:invariance_condition}
Let $\tau \in \mathcal{S}_n$ be an involution. Then
\[
w^\times(\tau \pi \tau^{-1}) = w^\times(\pi) \quad \text{for all } \pi \in \mathcal{S}_n
\]
if and only if $\tau$ is either the identity or the reversal permutation $\tau(i) = n+1-i$.
\end{theorem}

\begin{proof}
Since $w^\times$ depends only on distances between images of adjacent positions, any conjugation preserving cross work must map adjacent pairs to adjacent pairs. Thus $\tau$ induces an automorphism of the path graph $1 - 2 - \cdots - n$. The path graph admits exactly two automorphisms: the identity and the reversal. Conversely, reversal preserves cross work by Proposition~\ref{prop:reversal-invariance}.
\end{proof}

\begin{example}\label{ex:conj}
	For \(n=5\), let \(\tau=54321\in\mathcal S_5\) be the reversal permutation $\tau(i)=6-i$ (Proposition~\ref{prop:reversal-invariance}), and let \(\pi=24153\in\mathcal S_5\). Since \(\tau^{-1}=\tau\), we have \(\sigma=\tau\pi\tau^{-1}=31524\). A direct computation yields
	\[
	w^{\times}(\pi)=|1-4|+|2-2|+|2-1|+|3-4|+|3-5|+|4-1|+|4-3|+|5-5|=11,
	\]
	and similarly \(w^{\times}(\sigma)=11\), in agreement with Proposition~\ref{prop:reversal-invariance}.
\end{example}

We next determine the exact lower bound of the cross work function and characterize all permutations that attain it. To this end, we first recall the standard characterization of \emph{bandwidth-one permutations} (equivalently, tridiagonal permutation matrices), which will be used in the proof of Theorem~\ref{thm:minimizers}.

\begin{lemma}\label{lem:classification}
	Let $\sigma\in\mathcal S_m$ satisfy \(|\sigma(i)-i|\le1\), for
	all \(1\le i\le m\). Then 
	\[
	\sigma
	=
	\prod_{i\in I}(i\,\,i+1),
	\]
	where \(I\subseteq\{1,\ldots,m-1\}\), and \(|i-j|>1\) for all \(i\neq j\).
	
	In other words, every bandwidth-one permutation can be written as a product of pairwise disjoint adjacent transpositions.
\end{lemma}

A proof can be extracted from the arguments in Albert et al.~\cite[Sections~1--3]{AlbertLiStrangYu2011}, Samson and Ezerman~\cite[Theorem~7]{SamsonEzerman2012}, and Panova~\cite{Panova2010}.

\begin{theorem}\label{thm:minimizers}
Let $n \ge 3$ be an integer and $\alpha\in \mathcal{S}_n$. Then:
\begin{enumerate}
\item If $n=3$, then \(w^\times(\alpha) \ge 3\), with equality if and only if $\alpha\in \mathcal{N}_3=\{(1\,2),(2\,3),(1\,2\,3),(1\,3\,2)\}$.
\item If $n \ge 4$, then \(w^\times(\alpha) \ge 2(n-2)\), with equality if and only if 
\[\alpha \in 	\mathcal{N}_n
=
\left\{
s_1
\left(\prod_{i\in I}s_i\right)
s_{n-1}
:\;
I\subseteq\{3,4,\ldots,n-3\},\;
|i-j|>1\, \text{for all}\, i\neq j\in I
\right\},\]
\end{enumerate}
 where $s_i=(i\,i+1)$ denotes the adjacent transposition in $\mathcal S_n$ and the condition \(|i-j|>1\) ensures that the chosen adjacent transpositions are pairwise disjoint. In particular, the minimum cross work on $\mathcal{S}_n$ is
\[
\min_{\alpha\in \mathcal{S}_n} w^\times(\alpha) =
\begin{cases}
3, & n = 3, \\
2(n-2), & n \ge 4.
\end{cases}
\]
\end{theorem}

\begin{proof}
	By definition,
	\begin{equation}\label{eq:total_cross_work_Sn}
	w^\times(\alpha)
	=
	\sum_{i=1}^{n-1}|i-\alpha(i+1)|
	+	
	\sum_{i=1}^{n-1}|(i+1)-\alpha(i)|.
	\end{equation}
	Reindexing the first sum of \eqref{eq:total_cross_work_Sn} by $j=i+1$ and the second by $j=i$, we obtain the following useful positional form
	\begin{equation}\label{eq:reindexed}
	w^\times(\alpha)
	=
	\sum_{j=2}^{n-1}
	\Bigl(
	|(j-1)-\alpha(j)|
	+
	|(j+1)-\alpha(j)|
	\Bigr)
	+
	|2-\alpha(1)|
	+
	|(n-1)-\alpha(n)|.
	\end{equation}	
	For every integer $x$ and $j\in\{2,\ldots,n-1\}$,
	\[
	|(j-1)-x|+|(j+1)-x|
	\ge
	|(j+1)-(j-1)|
	=2,
	\]
	by the triangle inequality, with equality if and only if \(x\in\{j-1,j,j+1\}\). Applying this estimate to each summand of \eqref{eq:reindexed} yields
	\begin{equation}\label{eq:LB}
	w^\times(\alpha)
	\ge
	2(n-2)
	+
	|2-\alpha(1)|
	+
	|(n-1)-\alpha(n)|.
	\end{equation}
	
	\medskip
	
	Suppose first that $n=3$. Then
	\[
	w^\times(\alpha)
	\ge
	2
	+
	|2-\alpha(1)|
	+
	|2-\alpha(3)|.
	\]
	Since $\alpha$ is a bijection, $\alpha(1)\neq\alpha(3)$. Hence at least one of the two boundary terms is equal to~$1$, so
	\[
	|2-\alpha(1)|+|2-\alpha(3)|\ge1.
	\]
	Therefore \(w^\times(\alpha)\ge3\). Equality holds precisely when \(\{\alpha(1),\alpha(3)\}=\{1,2\}\) or \(\{2,3\}\), which yields exactly the four permutations \((1\,2),\; (2\,3),\; (1\,2\,3),\; (1\,3\,2)\).
	
	\medskip
	
	Assume now that $n\ge4$. Since the boundary terms in \eqref{eq:LB} are nonnegative,
	\[
	w^\times(\alpha)\ge2(n-2).
	\]	
	Suppose equality holds. Then equality must hold in every application of the triangle inequality. Consequently, \(\alpha(j)\in\{j-1,j,j+1\}\) for all \(2\le j\le n-1\), and \(\alpha(1)=2\), \(\alpha(n)=n-1\). Since $\alpha$ is a permutation, it follows that \(\alpha(2)=1\), \(\alpha(n-1)=n\). Hence
	\[
	|\alpha(j)-j|\le1,
	\]	
	for every \(1\le j\le n\). By Lemma~\ref{lem:classification}, $\alpha$ is a product of pairwise disjoint adjacent transpositions. Because \(\alpha(1)=2\) and \(\alpha(n)=n-1\), the transpositions \(s_1=(1\,2)\), \(s_{n-1}=(n-1\,n)\) must occur. Therefore
	\[
	\alpha
	=
	s_1
	\left(
	\prod_{i\in I}s_i
	\right)
	s_{n-1},
	\]
	where \(I\subseteq\{3,\ldots,n-3\}\), and \(|i-j|>1\) for \(i\neq j\).
	
	Conversely, every permutation of this form satisfies \(|\alpha(j)-j|\le1\) for all \(1\le j\le n\), together with \(\alpha(1)=2\), \(\alpha(n)=n-1\). Hence equality holds in every application of the triangle inequality above, and therefore \(w^\times(\alpha)=2(n-2)\). This completes the proof.
\end{proof}

\begin{lemma}\label{lem:bijection}
	The correspondence
	\[
	I
	\longmapsto
	s_1
	\left(\prod_{i\in I}s_i\right)
	s_{n-1},
	\]
	where \(I\subseteq\{3,\ldots,n-3\}\), and \(|i-j|>1\) for all \(i\neq j\), is injective.
\end{lemma}

\begin{proof}
	Since the adjacent transpositions indexed by $I$ are pairwise disjoint, they commute. Moreover, each chosen index produces a unique adjacent transposition, while every omitted index leaves the corresponding adjacent pair fixed. Hence, the resulting permutation uniquely determines the subset $I$, establishing injectivity.
\end{proof}

\begin{lemma}\label{lem:fibonacci}
	Let \(a_n:=|\mathcal{N}_n|\). Then
	\[
	a_n=a_{n-1}+a_{n-2},
	\]
	for \(n\ge6\), with initial conditions \(a_4=a_5=1\).
\end{lemma}

\begin{proof}
	There are exactly $n-5$ possible interior adjacent transpositions, \((3\,4), (4\,5),\ldots, (n-3\,n-2)\). Every admissible subset either
	
	\medskip\noindent\textit{Case (1).} does not contain the last adjacent transposition $(n-3\,n-2)$, in which case the remaining choices form an admissible subset corresponding to $a_{n-1}$; or
	
	\medskip\noindent\textit{Case (2).} contains $(n-3\,n-2)$, forcing $(n-4\,n-3)$ to be absent, so that the remaining choices correspond to an admissible subset counted by $a_{n-2}$.
	
	These two cases are mutually exclusive and exhaustive. Hence
	\[
	a_n=a_{n-1}+a_{n-2}.
	\]
	The initial values follow immediately from Theorem~\ref{thm:minimizers}: for $n=4$ and $n=5$ there is only one admissible subset, namely the empty set.
\end{proof}

\begin{proposition}\label{prop:mincount}
 	For every integer $n\ge4$, the number of minimizers is
	\[
	|\mathcal{N}_n|
	=
	F_{n-3},
	\]
	where $\{F_n\}_{n\ge1}$ is the Fibonacci sequence \(F_1=F_2=1\), \(F_n=F_{n-1}+F_{n-2}\) for \(n\ge3\).
\end{proposition}

The appearance of the Fibonacci sequence here is consistent with the
observation made in \cite[Remark~23]{SamsonEzerman2012}.

\begin{proof}
	The structural characterization of \(\mathcal{N}_n\) is given by Theorem~\ref{thm:minimizers}. By Lemma~\ref{lem:bijection}, every admissible subset of pairwise disjoint interior adjacent transpositions determines a unique minimizer. Lemma~\ref{lem:fibonacci} shows that the number of admissible subsets satisfies the Fibonacci recurrence with initial conditions \(a_4=a_5=1\), whence $a_n=F_{n-3}$ by uniqueness of the Fibonacci sequence.
\end{proof}

We now turn to the maximization problem for the cross work. In contrast with the minimum problem treated earlier, the maximization problem is global in nature: the contribution at one position depends not only on the value placed there, but also on how the remaining values are distributed. We therefore begin by establishing a structural constraint satisfied by every maximizer, and then optimize over the resulting class of permutations. 

For convenience, define
\[
F_j(x)=
\begin{cases}
|2-x|, & j=1,\\[4pt]
|j-1-x|+|j+1-x|, & 2\le j\le n-1,\\[4pt]
|n-1-x|, & j=n.
\end{cases}
\]
Then \eqref{eq:reindexed} may be rewritten as
\[
w^\times(\alpha)=\sum_{j=1}^n F_j(\alpha(j)).
\]

\begin{lemma}
	\label{lem:exchange}
	Let $1\le p<q\le n$ and let $a<b$. Suppose that a permutation $\alpha\in S_n$ satisfies
	\[
	\alpha(p)=a,\qquad \alpha(q)=b.
	\]
	Let $\alpha'$ be obtained from $\alpha$ by exchanging the values at the positions $p$ and $q$. Then
	\[
	w^\times(\alpha')-w^\times(\alpha)
	=
	\bigl(F_p(b)+F_q(a)\bigr)-\bigl(F_p(a)+F_q(b)\bigr)\ge 0.
	\]
	Moreover, the inequality is strict whenever $a$ and $b$ bracket the
	midpoint $n/2$ and are placed in inverted order, that is, $a$ occurs at the larger index $q$ and $b$ at the smaller index $p$. 
\end{lemma}

\begin{proof}
	For each fixed position $j$, the map $x\mapsto F_j(x)$ is convex and piecewise linear in $x$. More precisely, for $2\le j\le n-1$,
	\[
	F_j(x)=|j-1-x|+|j+1-x|,
	\]
	while for the boundary positions the one-sided analogues
	\[
	F_1(x)=|2-x|,\qquad F_n(x)=|n-1-x|
	\]
	hold. As $j$ increases, the center of the corresponding absolute-value profile moves to the right. It follows that for $p<q$, the difference
	\[
	x\longmapsto F_p(x)-F_q(x)
	\]
	is nonincreasing. Hence, for $a<b$,
	\[
	F_p(b)-F_q(b)\ge F_p(a)-F_q(a),
	\]
	which is equivalent to
	\[
	F_p(b)+F_q(a)\ge F_p(a)+F_q(b).
	\]
	This proves the first assertion. The inequality is strict when $a$ and $b$ are separated by the midpoint and are placed in misaligned positions, since in that case the exchange moves the larger value leftward and the smaller value rightward, thereby increasing both relevant distances simultaneously.
\end{proof}

The preceding lemma is the basic optimality mechanism, which is a classical Monge property \cite{BurkardKlinzRudolf1996} and their consequences for assignment and matching problems \cite{KorteVygen2018}. It implies that in any maximizing permutation, large values must occur in early positions and small values in late positions, respectively.

\begin{corollary}
	\label{cor:half-separation}
	Let $\mathcal M_n=\{\alpha\in \mathcal S_n: w^\times(\alpha)=\max_{\sigma\in \mathcal S_n} w^\times(\sigma)\}$. If $\alpha\in\mathcal M_n$, then the following hold.
	
	\begin{enumerate}
		\item If $n=2m$, then
		\[
		\alpha(\{1,\dots,m\})=\{m+1,\dots,2m\},
		\qquad
		\alpha(\{m+1,\dots,2m\})=\{1,\dots,m\}.
		\]
		
		\item If $n=2m+1$, then
		\[
		\alpha(\{1,\dots,m\})=\{m+2,\dots,2m+1\},
		\qquad
		\alpha(\{m+2,\dots,2m+1\})=\{1,\dots,m\},
		\]
		and necessarily
		\[
		\alpha(m+1)=m+1.
		\]
	\end{enumerate}
\end{corollary}

\begin{proof}
	Assume first that $n=2m$. If there exist indices $p\le m$ and $q\ge m+1$ such that
	\[
	\alpha(p)\le m,\qquad \alpha(q)\ge m+1,
	\]
	then $\alpha(p)<\alpha(q)$, and by Lemma~\ref{lem:exchange}, swapping these two values strictly increases $w^\times(\alpha)$. This contradicts maximality. Therefore every position in the first half must receive a value from the second half, and conversely.
	
	Now suppose that $n=2m+1$. The same exchange argument shows that the two outer blocks must be completely separated:
	\[
	\alpha(\{1,\dots,m\})=\{m+2,\dots,2m+1\},
	\qquad
	\alpha(\{m+2,\dots,2m+1\})=\{1,\dots,m\}.
	\]
	The only remaining value is $m+1$, so injectivity forces it to occur at the remaining position, namely $m+1$. Hence $\alpha(m+1)=m+1$.
\end{proof}

Corollary~\ref{cor:half-separation} reduces the maximization problem to a much smaller class of permutations. On this half-separated region, the absolute values in \eqref{eq:reindexed} acquire fixed signs, and the cross work becomes linear.

\begin{proposition}
	\label{prop:even-linear}
	Assume that $n=2m$ and that $\alpha\in \mathcal S_{2m}$ satisfies
	\[
	\alpha(\{1,\dots,m\})=\{m+1,\dots,2m\},
	\qquad
	\alpha(\{m+1,\dots,2m\})=\{1,\dots,m\}.
	\]
	Then
	\[
	w^\times(\alpha)=4m^2-2m-1-\alpha(1)+\alpha(2m).
	\]
\end{proposition}

\begin{proof}
	By the half-separation hypothesis, for each $2\le j\le m$ we have
	\[
	\alpha(j)\ge m+1\ge j+1,
	\]
	and hence
	\[
	|j-1-\alpha(j)|+|j+1-\alpha(j)|=2\alpha(j)-2j.
	\]
	Likewise, for each $m+1\le j\le 2m-1$ we have
	\[
	\alpha(j)\le m\le j-1,
	\]
	and therefore
	\[
	|j-1-\alpha(j)|+|j+1-\alpha(j)|=2j-2\alpha(j).
	\]
	The boundary terms satisfy
	\[
	|2-\alpha(1)|=\alpha(1)-2,\qquad |2m-1-\alpha(2m)|=2m-1-\alpha(2m).
	\]
	Substituting these identities into \eqref{eq:reindexed} gives the following
	\[
	w^\times(\alpha)
	=
	(\alpha(1)-2)
	+\sum_{j=2}^{m}(2\alpha(j)-2j)
	+\sum_{j=m+1}^{2m-1}(2j-2\alpha(j))
	+(2m-1-\alpha(2m)).
	\]
	To derive the required formula, we use
	\[
	\sum_{j=1}^{m}\alpha(j)=\sum_{r=m+1}^{2m}r,
	\qquad
	\sum_{j=m+1}^{2m}\alpha(j)=\sum_{r=1}^{m}r.
	\]
	After collecting terms and simplifying, we obtain
	\[
	w^\times(\alpha)=4m^2-2m-1-\alpha(1)+\alpha(2m),
	\]
	as claimed.
\end{proof}

\begin{corollary}
	\label{cor:even-max}
	Suppose that $n=2m$. Then a half-separated permutation $\alpha\in \mathcal S_{2m}$ maximizes $w^\times$ if and only if
	\[
	\alpha(1)=m+1,\qquad \alpha(2m)=m,
	\]
	and
	\[
	\alpha(\{2,\dots,m\})=\{m+2,\dots,2m\},
	\qquad
	\alpha(\{m+1,\dots,2m-1\})=\{1,\dots,m-1\}.
	\]
	In that case,
	\[
	w^\times(\alpha)=4m^2-2m-2=(2m)^2-2m-2=n^2-n-2,
	\]
	and
	\[
	|\mathcal M_{2m}|=((m-1)!)^2.
	\]
\end{corollary}

\begin{proof}
	By Proposition~\ref{prop:even-linear},
	\[
	w^\times(\alpha)=4m^2-2m-1-\alpha(1)+\alpha(2m).
	\]
	Under the half-separation condition from Corollary~\ref{cor:half-separation}, we have
	\[
	\alpha(1)\in\{m+1,\dots,2m\},\qquad \alpha(2m)\in\{1,\dots,m\}.
	\]
	Therefore the right-hand side is maximal exactly when $\alpha(1)$ is as small as possible and $\alpha(2m)$ is as large as possible, namely
	\[
	\alpha(1)=m+1,\qquad \alpha(2m)=m.
	\]
	Substituting these values yields
	\[
	w^\times(\alpha)=4m^2-2m-2=n^2-n-2.
	\]
	Once the boundary values are fixed, the remaining $m-1$ large values may be arranged arbitrarily in the positions $2,\dots,m$, and the remaining $m-1$ small values may be arranged arbitrarily in the positions $m+1,\dots,2m-1$. Hence
	\[
	|\mathcal M_{2m}|=(m-1)!(m-1)!=((m-1)!)^2.
	\]
\end{proof}

\begin{proposition}
	\label{prop:odd-linear}
	Assume that $n=2m+1$ and that $\alpha\in \mathcal S_{2m+1}$ satisfies
	\[
	\alpha(\{1,\dots,m\})=\{m+2,\dots,2m+1\},
	\quad
	\alpha(m+1)=m+1,
	\quad
	\alpha(\{m+2,\dots,2m+1\})=\{1,\dots,m\}.
	\]
	Then
	\[
	w^\times(\alpha)=4m^2+2m-\alpha(1)+\alpha(2m+1).
	\]
\end{proposition}

\begin{proof}
	By Corollary~\ref{cor:half-separation}, the middle value is fixed:
	\[
	\alpha(m+1)=m+1.
	\]
	Exactly as in the proof of Proposition~\ref{prop:even-linear}, the separated outer blocks force fixed signs in the interior terms of \eqref{eq:reindexed}, and hence all such terms simplify linearly. Summing over the left and right outer blocks and using
	\[
	\sum_{j=1}^{m}\alpha(j)=\sum_{r=m+2}^{2m+1}r,
	\qquad
	\sum_{j=m+2}^{2m+1}\alpha(j)=\sum_{r=1}^{m}r,
	\]
	one obtains
	\[
	w^\times(\alpha)=4m^2+2m-\alpha(1)+\alpha(2m+1).
	\]
\end{proof}

\begin{corollary}
	\label{cor:odd-max}
	Suppose that $n=2m+1$. Then a permutation $\alpha\in \mathcal S_{2m+1}$ maximizes $w^\times$ if and only if
	\[
	\alpha(1)=m+2,\qquad \alpha(m+1)=m+1,\qquad \alpha(2m+1)=m,
	\]
	and
	\[
	\alpha(\{2,\dots,m\})=\{m+3,\dots,2m+1\},
	\qquad
	\alpha(\{m+2,\dots,2m\})=\{1,\dots,m-1\}.
	\]
	In that case,
	\[
	w^\times(\alpha)=4m^2+2m-2=(2m+1)^2-(2m+1)-2=n^2-n-2,
	\]
	and
	\[
	|\mathcal M_{2m+1}|=((m-1)!)^2.
	\]
\end{corollary}

\begin{proof}
	Once Corollary~\ref{cor:half-separation} has reduced the problem to the half-separated class and fixed the middle point, Proposition~\ref{prop:odd-linear} shows that
	\[
	w^\times(\alpha)=4m^2+2m-\alpha(1)+\alpha(2m+1).
	\]
	Thus $w^\times(\alpha)$ is maximal exactly when $\alpha(1)$ is minimized and $\alpha(2m+1)$ is maximized subject to the half-separation constraint. Hence
	\[
	\alpha(1)=m+2,\qquad \alpha(2m+1)=m.
	\]
	Substituting gives
	\[
	w^\times(\alpha)=4m^2+2m-2=n^2-n-2.
	\]
	The remaining $m-1$ values in each outer block may be permuted independently, and therefore
	\[
	|\mathcal M_{2m+1}|=((m-1)!)^2.
	\]
\end{proof}

We may now combine Corollaries~\ref{cor:half-separation}, \ref{cor:even-max}, and \ref{cor:odd-max} to obtain the main extremal result of this section.

\begin{theorem}
	\label{thm:max-count}
	Let
	\[
	\mathcal M_n=\{\alpha\in \mathcal S_n:\; w^\times(\alpha)=n^2-n-2\}.
	\]
	Then for every integer $n\ge 3$,
	\[
	\max_{\alpha\in\mathcal S_n} w^\times(\alpha)=n^2-n-2.
	\]
	
	More precisely, the maximizers are exactly the following.
	
	\begin{enumerate}
		\item If $n=2m$, then
		\[
		\mathcal M_{2m}
		=
		\Bigl\{
		\alpha\in\mathcal S_{2m}:
		\alpha(1)=m+1,\ \alpha(2m)=m,\ 
		\alpha(\{2,\dots,m\})=\{m+2,\dots,2m\},
		\]
		\[
		\hspace{4.3cm}
		\alpha(\{m+1,\dots,2m-1\})=\{1,\dots,m-1\}
		\Bigr\},
		\]
		and hence
		\[
		|\mathcal M_{2m}|=((m-1)!)^2.
		\]
		
		\item If $n=2m+1$, then
		\[
		\mathcal M_{2m+1}
		=
		\Bigl\{
		\alpha\in\mathcal S_{2m+1}:
		\alpha(1)=m+2,\ \alpha(m+1)=m+1,\ \alpha(2m+1)=m,
		\]
		\[
		\hspace{2.4cm}
		\alpha(\{2,\dots,m\})=\{m+3,\dots,2m+1\},\
		\alpha(\{m+2,\dots,2m\})=\{1,\dots,m-1\}
		\Bigr\},
		\]
		and hence
		\[
		|\mathcal M_{2m+1}|=((m-1)!)^2.
		\]
	\end{enumerate}
\end{theorem}

\begin{proof}
	The lower bound
	\[
	\max_{\alpha\in\mathcal S_n} w^\times(\alpha)\ge n^2-n-2
	\]
	is realized by the explicit block-swap permutations described in Corollaries~\ref{cor:even-max} and \ref{cor:odd-max}. Conversely, Corollary~\ref{cor:half-separation} shows that every maximizer must be completely half-separated, with the middle point fixed in the odd case. Propositions~\ref{prop:even-linear} and \ref{prop:odd-linear} then reduce the problem to a linear optimization in the boundary values, while Corollaries~\ref{cor:even-max} and \ref{cor:odd-max} identify exactly when that linear expression is maximal. In both parity cases, the maximal value is
	\[
	n^2-n-2.
	\]
	The descriptions of $\mathcal M_{2m}$ and $\mathcal M_{2m+1}$, together with their cardinalities, follow immediately from the two corollaries.
\end{proof}

\begin{remark}\label{rem:max-count}
	For $n\ge 4$, Theorem~\ref{thm:max-count} yields the uniform counting formula
	\[
	|\mathcal M_n|
	=
	\left(\left\lfloor\frac{n}{2}\right\rfloor-1\right)!
	\left(\left\lfloor\frac{n}{2}\right\rfloor-1\right)!.
	\]
	Since the two factorials coincide in both parity cases, this may be written more compactly as
	\[
	|\mathcal M_n|=((\lfloor n/2\rfloor-1)!)^2.
	\]
	For the exceptional case $n=3$, one has \(\mathcal M_3=\{123,321\}\)
	with \(|\mathcal M_3|=2\).
\end{remark}

\medskip

\begin{theorem}\label{thm:Sn}
For every integer $n \ge 2$, the total cross work on the symmetric group $\mathcal{S}_n$ is
\[
w^\times(\mathcal{S}_n) = 2(n-1)! \left[ \binom{n+1}{3} + \binom{n}{3} \right].
\]
\end{theorem}

\begin{proof}
By definition, we have
\begin{align*}
w^{\times}(\mathcal{S}_n) &= \sum_{\alpha \in \mathcal{S}_n} w^{\times}(\alpha) \\
&= \sum_{\alpha \in \mathcal{S}_n} \sum_{i=1}^{n-1}(|i-\alpha(i+1)| + |i+1 - \alpha(i)|) \\
&= \sum_{\alpha \in \mathcal{S}_n} \sum_{i=1}^{n-1}|i-\alpha(i+1)| + \sum_{\alpha \in \mathcal{S}_n} \sum_{i=1}^{n-1}|i+1 - \alpha(i)|.
\end{align*}
First, we compute \(\displaystyle \sum_{\alpha \in \mathcal{S}_n} \sum_{i=1}^{n-1}|i-\alpha(i+1)| \). By a discrete Fubini argument, we interchange the order of the sums to get
\[\displaystyle \sum_{\alpha \in \mathcal{S}_n} \sum_{i=1}^{n-1}(|i-\alpha(i+1)|) = \sum_{i=1}^{n-1}\sum_{\alpha \in \mathcal{S}_n}(|i-\alpha(i+1)|). \] 
For fixed $i$, the value $\alpha(i+1)$ ranges over all $j\in[n]$, and for each $j$ there are $(n-1)!$ permutations in \(\mathcal{S}_n\) with $\alpha(i+1)=j$. Thus,
\begin{align*}
\sum_{i=1}^{n-1}\sum_{\alpha \in \mathcal{S}_n} |i-\alpha(i+1)| &= \sum_{i=1}^{n-1} (n-1)! \sum_{j=1}^{n}|i-j| \\
&= (n-1)! \sum_{i=1}^{n-1} \sum_{j=1}^{n}|i-j| \\
&= (n-1)! \sum_{i=1}^{n-1} \left( \sum_{j=1}^{i}(i-j) + \sum_{j=i+1}^{n} (j-i) \right) \\
&= (n-1)! \sum_{i=1}^{n-1} \left( \frac{i(i-1)}{2} + \frac{(n-i)(n-i+1)}{2} \right) \\
&= (n-1)! \sum_{i=1}^{n-1} \left( \binom{i}{2} + \binom{n-i+1}{2} \right) \\
&= (n-1)! \left[ \sum_{i=1}^{n-1} \binom{i}{2} + \sum_{i=1}^{n-1}\binom{n-i+1}{2} \right].
\end{align*}
By the hockey-stick identity, \(\sum_{i=r}^{n}\binom{i}{r} = \binom{n+1}{r+1}\). Thus,
\[
\sum_{i=1}^{n-1}\sum_{\alpha \in \mathcal{S}_n} |i-\alpha(i+1)| = (n-1)! \left[ \binom{n+1}{3} + \binom{n}{3} \right].
\]
By a similar argument,
\[
\sum_{i=1}^{n-1}\sum_{\alpha \in \mathcal{S}_n} |i+1-\alpha(i)| = (n-1)! \left[ \binom{n+1}{3} + \binom{n}{3} \right].
\]
Adding these two expressions gives
\begin{align*}
w^{\times}(\mathcal{S}_n) &= 2(n-1)! \left[ \binom{n+1}{3} + \binom{n}{3} \right].
\end{align*}
\end{proof}

\begin{corollary}\label{cor:avg-crosswork}
The average cross work on $\mathcal{S}_n$ is
\[
\overline{w}^\times(\mathcal{S}_n) = \frac{2}{n} \left[ \binom{n+1}{3} + \binom{n}{3} \right].
\]
\end{corollary}

\begin{proof}
Since $|\mathcal{S}_n| = n!$, by Definition~\ref{def:total_crosswork} we have
\[
\overline{w}^\times(\mathcal{S}_n) = \frac{w^\times(\mathcal{S}_n)}{n!} = \frac{2(n-1)!}{n!} \left[ \binom{n+1}{3} + \binom{n}{3} \right] = \frac{2}{n} \left[ \binom{n+1}{3} + \binom{n}{3} \right].
\]
\end{proof}

Next we extend these results to other subsemigroups of the partial transformation semigroup.

\begin{theorem}\label{thm:Tn}
For $n \ge 2$, the total cross work on the full transformation semigroup $\mathcal{T}_n$ is
\[
w^\times(\mathcal{T}_n) = 2 n^{\,n-1} \left[ \binom{n+1}{3} + \binom{n}{3} \right].
\]
\end{theorem}

\begin{proof}
Fix $i \in \{1,\dots,n-1\}$. For each $j \in [n]$, there are exactly $n^{n-1}$ transformations $\alpha \in \mathcal{T}_n$ with $\alpha(i+1)=j$. Hence
\[
\sum_{\alpha \in \mathcal{T}_n} |i-\alpha(i+1)| = n^{n-1} \sum_{j=1}^n |i-j|.
\]
Summing over $i$ and using the same binomial identities as before gives
\[
\sum_{i=1}^{n-1} \sum_{\alpha \in \mathcal{T}_n} |i-\alpha(i+1)| = n^{n-1} \left[ \binom{n}{3} + \binom{n+1}{3} \right].
\]
The term $|i+1-\alpha(i)|$ contributes equally, yielding the stated formula.
\end{proof}

\begin{theorem}\label{thm:Sing}
For $n \ge 2$, the total cross work on the semigroup of singular transformations $\mathrm{Sing}_n$ is
\[
w^\times(\mathrm{Sing}_n) = 2\bigl(n^{\,n-1} - (n-1)!\bigr) \left[ \binom{n+1}{3} + \binom{n}{3} \right].
\]
\end{theorem}

\begin{proof}
Since $\mathcal{T}_n = \mathcal{S}_n \sqcup \mathrm{Sing}_n$,
\[
w^\times(\mathrm{Sing}_n) = w^\times(\mathcal{T}_n) - w^\times(\mathcal{S}_n).
\]
Substituting the formulas from Theorems~\ref{thm:Sn} and~\ref{thm:Tn} yields the result.
\end{proof}

\begin{theorem}\label{thm:I}
For $n \ge 2$,
\[
w^\times(\mathcal{I}_n) = 2 \left( \sum_{m=0}^{n-2} \binom{n-2}{m} \frac{(n-1)!}{(n-m-2)!} \right) \left[ \binom{n}{3} + \binom{n+1}{3} \right].
\]
\end{theorem}

\begin{proof}
Write
\[
w^\times(\mathcal{I}_n) = \sum_{\alpha \in \mathcal{I}_n} \sum_{i=1}^{n-1} \bigl(\delta_i^L(\alpha) + \delta_i^R(\alpha)\bigr).
\]
Define
\[
S_1 = \sum_{i=1}^{n-1} \sum_{\substack{\alpha \in \mathcal{I}_n \\ i,i+1 \in \operatorname{dom}(\alpha)}} |i-\alpha(i+1)|, \quad
S_2 = \sum_{i=1}^{n-1} \sum_{\substack{\alpha \in \mathcal{I}_n \\ i,i+1 \in \operatorname{dom}(\alpha)}} |i+1-\alpha(i)|.
\]
There is an involution $\varphi$ on the set of injections $\alpha$ with $i,i+1 \in \operatorname{dom}(\alpha)$ defined by swapping $\alpha(i)$ and $\alpha(i+1)$. Under $\varphi$, the term $|i-\alpha(i+1)|$ for $\alpha$ equals $|i+1-\varphi(\alpha)(i)|$ for $\varphi(\alpha)$. Summing over all such $\alpha$ gives $S_1 = S_2$, hence $w^\times(\mathcal{I}_n) = 2S_1$.

Now fix $i$. For each $j \in [n]$, let
\[
C_n(i,j) = \bigl| \{ \alpha \in \mathcal{I}_n : i,i+1 \in \operatorname{dom}(\alpha),\; \alpha(i+1)=j \} \bigr|.
\]
Then
\[
\sum_{\substack{\alpha \in \mathcal{I}_n \\ i,i+1 \in \operatorname{dom}(\alpha)}} |i-\alpha(i+1)| = \sum_{j=1}^n |i-j| \, C_n(i,j).
\]
We claim $C_n(i,j)$ is independent of $i$ and $j$. Indeed, given $i$ and $j$, choose $\alpha(i)$ from $[n]\setminus\{j\}$ ($n-1$ choices). Then, for each \( m \) with \( 0 \le m \le n-2 \), choose any subset $D' \subseteq [n]\setminus\{i,i+1\}$ of size $m$ to be the rest of the domain (there are \( \binom{n-2}{m} \) such subsets), and injectively map $D'$ to $[n]\setminus\{\alpha(i),j\}$ in $\frac{(n-2)!}{(n-m-2)!}$ ways. Hence
\[
C_n(i,j) = \sum_{m=0}^{n-2} \binom{n-2}{m} \frac{(n-2)!}{(n-m-2)!} (n-1),
\]
an expression that clearly depends on neither $i$ nor $j$. Denote this common value by $C_n$. Then
\[
S_1 = C_n \sum_{i=1}^{n-1} \sum_{j=1}^n |i-j| = C_n \left[ \binom{n+1}{3} + \binom{n}{3} \right].
\]
Let $T_n(i) = \bigl| \{ \alpha \in \mathcal{I}_n : i,i+1 \in \operatorname{dom}(\alpha) \} \bigr|$. Clearly, $T_n(i) = n C_n$. Alternatively, we can compute \( T_n(i) \) directly. For a fixed \( i \), choose a domain \( D \supseteq \{i,i+1\} \) of size \( k \) where \( 2 \le k \le n \). The number of such domains is \( \binom{n-2}{k-2} \). For each \( D \), the number of injections \( \alpha : D \to [n] \) is \( \dfrac{n!}{(n-k)!} \). Hence
\[
T_n(i) = \sum_{k=2}^n \binom{n-2}{k-2} \frac{n!}{(n-k)!}.
\]
Setting \( m = k-2 \) yields
\[
T_n(i) = \sum_{m=0}^{n-2} \binom{n-2}{m} \frac{n!}{(n-m-2)!} =: T_n.
\]
Thus $C_n = T_n/n$, and
\[
w^\times(\mathcal{I}_n) = \frac{2T_n}{n} \left[ \binom{n+1}{3} + \binom{n}{3} \right].
\]
Substituting $T_n$ yields the stated formula.
\end{proof}

\begin{lemma}\label{lem:N}
Fix $i \in [n-1]$ and $1 \le a < b \le n$. The number of $\alpha \in \mathcal{POI}_n$ with $i,i+1 \in \operatorname{dom}(\alpha)$, $\alpha(i)=a$, $\alpha(i+1)=b$ is
\[
N(i,a,b) = \binom{i+a-2}{i-1} \binom{2n-i-b-1}{n-i-1}.
\]
\end{lemma}

\begin{proof}
Let $D = \operatorname{dom}(\alpha)$. Since $\alpha$ is order-preserving, if $x \in D$ with $x < i$, then $\alpha(x) < a$; if $x \in D$ with $x > i+1$, then $\alpha(x) > b$.

Let $m_1 = |\{x \in D : x < i\}|$ and $m_2 = |\{x \in D : x > i+1\}|$. Choosing $m_1$ elements from $\{1,\dots,i-1\}$ and their images from $\{1,\dots,a-1\}$ in increasing order gives $\binom{i-1}{m_1} \binom{a-1}{m_1}$ possibilities. Similarly, choosing $m_2$ elements from $\{i+2,\dots,n\}$ and their images from $\{b+1,\dots,n\}$ gives $\binom{n-i-1}{m_2} \binom{n-b}{m_2}$ possibilities. Thus the number for given $(m_1,m_2)$ is $\binom{i-1}{m_1}\binom{a-1}{m_1} \binom{n-i-1}{m_2}\binom{n-b}{m_2}$. Summing over $m_1,m_2$ and applying Vandermonde's identity $\sum_{k=0}^r \binom{r}{k}\binom{s}{k} = \binom{r+s}{r}$ yields
\[
\sum_{m_1=0}^{i-1} \binom{i-1}{m_1}\binom{a-1}{m_1} = \binom{i+a-2}{i-1},
\]
\[
\sum_{m_2=0}^{n-i-1} \binom{n-i-1}{m_2}\binom{n-b}{m_2} = \binom{2n-i-b-1}{n-i-1}.
\]
Multiplying gives the desired result.
\end{proof}

\begin{theorem}\label{thm:POI}
For every integer $n \ge 2$,
\[
w^\times(\mathcal{POI}_n) = 2 \sum_{i=1}^{n-1} \sum_{1 \le a < b \le n} |i-b| \cdot \binom{i+a-2}{i-1} \binom{2n-i-b-1}{n-i-1}.
\]
\end{theorem}

\begin{proof}
For $\alpha \in \mathcal{POI}_n$, the term $|i-\alpha(i+1)|$ contributes for each $i$ with $i,i+1 \in \operatorname{dom}(\alpha)$. Summing over all such $\alpha$ gives
\[
S_1 = \sum_{i=1}^{n-1} \sum_{1 \le a < b \le n} |i-b| \cdot N(i,a,b),
\]
where $a=\alpha(i)$, $b=\alpha(i+1)$. Similarly,
\[
S_2 = \sum_{i=1}^{n-1} \sum_{1 \le a < b \le n} |i+1-a| \cdot N(i,a,b).
\]
By symmetry, $S_1 = S_2$. Hence $w^\times(\mathcal{POI}_n) = 2S_1$. Substituting $N(i,a,b)$ from Lemma~\ref{lem:N} gives the result.
\end{proof}

\begin{remark}
	For completeness, Table~\ref{tab:values} records the values of \(w^{\times}\) on the semigroups considered above for \(2\le n\le 10\); the fifth column is \(w^{\times}(\mathcal I_n)\) and the sixth is \(w^{\times}(\mathcal{POI}_n)\). 

	\begin{table}[ht]
	\centering
	\caption{Total cross work on the semigroups
		$\mathcal S_n$, $\mathcal T_n$, $\mathrm{Sing}_n$, $\mathcal I_n$,
		and $\mathcal{POI}_n$ for $2 \le n \le 10$}
	\label{tab:values}
	\begin{tabular}{c|r|r|r|r|r}
	$n$
	& $w^\times(\mathcal{S}_n)$
	& $w^\times(\mathcal{T}_n)$
	& $w^\times(\mathrm{Sing}_n)$
	& $w^\times(\mathcal{I}_n)$
	& $w^\times(\mathcal{POI}_n)$ \\
	\hline
	2 & 2 & 4 & 2 & 2 & 2 \\
	3 & 20 & 90 & 70 & 40 & 18 \\
	4 & 168 & 1,792 & 1,624 & 588 & 112 \\
	5 & 1,440 & 37,500 & 36,060 & 8,160 & 604 \\
	6 & 13,200 & 855,360 & 842,160 & 114,950 & 3,024 \\
	7 & 131,040 & 21,412,118 & 21,281,078 & 1,688,232 & 14,474 \\
	8 & 1,411,200 & 587,202,560 & 585,791,360 & 26,120,920 & 67,232 \\
	9 & 16,450,560 & 17,563,062,168 & 17,546,611,608 & 427,329,408 & 305,700 \\
	10 & 206,841,600 & 570,000,000,000 & 569,793,158,400 & 7,396,069,770 & 1,368,016
	\end{tabular}
	\vspace{2mm}
	
	\begin{minipage}{0.8\linewidth}
		\footnotesize
		All entries independently verified by brute-force enumeration for \(2 \le n \le 8\) and by the formulas of \S~\ref{sec:main_results} for \(2 \le n \le 10\).
	\end{minipage}
	\end{table}
	
	\begin{figure}[htbp]
	\centering
	\vspace{-80pt}
	\includegraphics[width=0.85\textwidth]{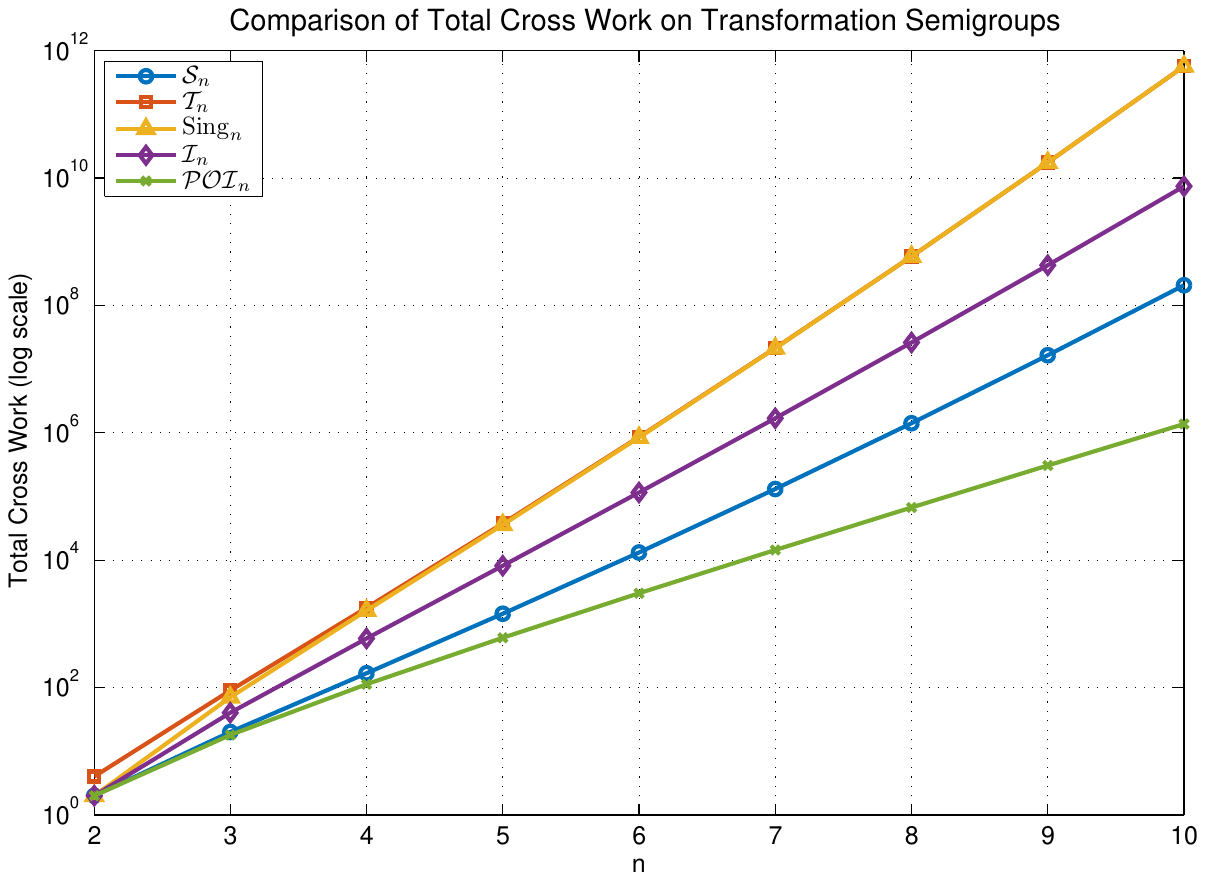}
	\vspace{-105pt}
	\caption{Comparison of total cross work on various transformation semigroups for $2 \le n \le 10$.}
	\label{fig:comparison}
	\end{figure}
\end{remark}

\begin{remark}
Figure~\ref{fig:comparison} illustrates several structural phenomena. The total cross work on $\mathcal{T}_n$ and $\mathrm{Sing}_n$ grows super-exponentially, with $\mathrm{Sing}_n$ asymptotically dominating the full semigroup. The symmetric inverse semigroup $\mathcal{I}_n$ exhibits comparable growth, reflecting its factorial rank structure. In contrast, $\mathcal{POI}_n$ grows polynomially-exponentially, highlighting the rigidity imposed by order preservation.
\end{remark}

\section*{Concluding remarks}
We introduced the cross work as a new numerical invariant on the symmetric group $\mathcal{S}_n$, the semigroup of singular transformations $\mathrm{Sing}_n$, the symmetric inverse semigroup $\mathcal{I}_n$, and the monoid $\mathcal{POI}_n$ of partial order-preserving injective transformations. This statistic captures local adjacency-based displacement induced by finite transformations. We established its fundamental invariance properties, including inversion invariance under a natural symmetry condition and a complete characterization of the conjugations preserving the statistic. We also derived closed formulas for the total and average cross work on $\mathcal{S}_n$, and determined its exact minimum and maximum values. These results show that the minimum and maximum cross work grow linearly and quadratically with $n$, respectively, while the numbers of minimizing and maximizing permutations are given by the classical OEIS sequences \cite[A000045]{OEIS2024} and \cite[A180064]{OEIS2024}, respectively (with the appropriate indexing). Taken together, these results establish the cross work as a well-defined and computable invariant arising naturally from transformation semigroup theory, and they provide a basis for further investigations of adjacency-based statistics on permutations and more general transformation semigroups.\\

\noindent\textbf{Acknowledgements:}
The authors thank Bal\'azs Szendr\"oi and M.~Dolors Magret for their valuable comments that improved this work.

\vspace{2em}

\noindent
\textsc{University of Calabar, Nigeria} \\
\textit{Email address:} \texttt{nsibietudo@unical.edu.ng}

\vspace{1em}

\noindent
\textsc{University of Abuja, Nigeria} \\
\textit{Email address:} \texttt{owaiakugom@gmail.com}

\vspace{1em}

\noindent
\textsc{Bingham University, Kari, Nigeria} \\
\textit{Email address:} \texttt{francis-moses.obinna@binghamuni.edu.ng}


\begin{thebibliography}{99}

\bibitem{AlbertLiStrangYu2011}
 C.~Albert, C.~K.~Li, G.~Strang and G.~Yu, 
 \textit{Permutations as product of parallel transpositions}, SIAM Journal on Discrete Mathematics, 25(3), (2011), 1412-1417.

\bibitem{Bona2012}
M.~Bóna,
\textit{Combinatorics of Permutations},
Second Edition,
CRC Press, Boca Raton, (2012).

\bibitem{BurkardKlinzRudolf1996}
R.~E.~Burkard, B.~Klinz and R.~Rudolf,
\textit{Perspectives of {M}onge properties in optimization},
Discrete Applied Mathematics, 70(2), (1996), 95--161.

\bibitem{DalyVojtechovsky2009}
D.~Daly and P.~Vojtěchovský,
\textit{How permutations displace points and stretch intervals},
Ars Combinatoria,
90 (2009), 175--191.

\bibitem{DiaconisGraham1977}
P.~Diaconis and R.~L.~Graham,
\textit{Spearman's footrule as a measure of disarray},
Journal of the Royal Statistical Society, Series~B,
39(2), (1977), 262--268.

\bibitem{EastMcNamara2011}
J.~East and P.~J.~McNamara,
\textit{On the work performed by a transformation semigroup},
Australasian Journal of Combinatorics,
49 (2011), 95--109.

\bibitem{FrancisAdenijiMogbonju2023}
M.~O.~Francis, A.~O.~Adeniji, and M.~M.~Mogbonju,
\textit{Work done by $m$-topological transformation semigroup regular spaces $M_{\psi n}$},
International Journal of Mathematical Sciences and Optimization,
9(1), (2023), 33--42.

\bibitem{FrancisAdeniji2025}
M.~O.~Francis and A.~O.~Adeniji,
\textit{Work performed by closed and clopen $m$-topological full transformation semigroup spaces},
Unilag Journal of Mathematics and Applications,
5(1), (2025), 51--62.

\bibitem{Howie1995}
J.~M.~Howie,
\textit{Fundamentals of Semigroup Theory},
London Mathematical Society Monographs,
Oxford University Press, Oxford, (1995).

\bibitem{Kehinde.et.al.2020}
R.~Kehinde, I.~L.~David, A.~I.~Ma'ali, A.~Abdulrahman, and O.~H.~Abdulazeez,
\textit{The numerical work done by transformations on a symmetric group},
International Journal of Innovation in Science and Mathematics,
8(5), (2020), 252--258.

\bibitem{KorteVygen2018}
B.~Korte and J.~Vygen,
\textit{Combinatorial Optimization: Theory and Algorithms}, 6 edn., Springer, (2018).

\bibitem{LaradjiUmar2004}
A.~Laradji and A.~Umar,
\textit{Combinatorial results for semigroups of order-preserving partial transformations},
Journal of Algebra,
278 (2004), 342--359.

\bibitem{Lavers2004}
T.~Lavers,
\textit{Semigroup invariants and transformation complexity},
Semigroup Special Interest Meeting, Sydney, (2004) (unpublished lecture).

\bibitem{MacMahon1915}
P.~A.~MacMahon,
\textit{Combinatory Analysis, Vol.~1},
Cambridge University Press, Cambridge, (1915).

\bibitem{Panova2010} G.~Panova, \textit{Factorization of banded permutations}, Proceedings of the American Mathematical Society, 140(11), (2012), 3805-3812.\href{https://arxiv.org/abs/1007.1760v2}{arXiv:1007.1760v2}

\bibitem{SamsonEzerman2012} M.~D.~Samson and M.~F.~Ezerman, \textit{Factoring Permutation Matrices Into a Product of Tridiagonal Matrices}, (2010) \href{https://arxiv.org/abs/1007.3467v1}{arXiv:1007.3467v1}.

\bibitem{OEIS2024} N.~J.~A.~Sloane, The On-Line Encyclopedia of Integer Sequences, Available at \url{https://oeis.org}.

\bibitem{Stanley2012}
R.~P.~Stanley,
\textit{Enumerative Combinatorics, Volume~1},
Second Edition,
Cambridge University Press, Cambridge, (2012).

\bibitem{Umar2016}
A.~Umar,
\textit{Some combinatorial problems in the theory of partial transformation semigroups},
Algebra and Discrete Mathematics,
17(1), (2014), 1--26.

\end{thebibliography}
\end{document}